\documentclass[a4paper,11pt]{article}

\usepackage[small,bf,hang]{caption} 

\usepackage{todonotes}
\usepackage[ruled,linesnumbered]{algorithm2e}
\usepackage{amsmath,amsthm,amssymb}
\usepackage{graphicx,subcaption,tikz}
\usepackage{comment,xspace,paralist}							
\usepackage[shortlabels]{enumitem}	
\usepackage{hyperref}
\usepackage{changepage}
\hypersetup{
    colorlinks=true,
    citecolor = blue,
    linkcolor=black,
    filecolor=magenta,
    urlcolor=cyan,
    bookmarks=true,
}
\usepackage[margin=1in]{geometry}

\newtheorem{theorem}{Theorem}
\newtheorem{lemma}{Lemma}
\newtheorem{claim}{Claim}

\newtheorem{problem}{Problem}
\newtheorem{definition}{Definition}

\newcommand{\case}[1]{\noindent {\bf Case #1.}}

\title{{\bf Near Optimal Colourability on Hereditary Graph Families}}

\author{
Yiao Ju\thanks{College of Computer Science, Nankai University, Tianjin 300350, P. R. China.}
\and
Shenwei Huang\thanks{College of Computer Science, Nankai University, Tianjin 300350, P. R. China. Supported by NSFC under Grant 12171256.}
}

\begin{document}

\maketitle

\begin{abstract}
In this paper, we initiate a systematic study on a new notion called near optimal colourability which is closely related to
perfect graphs and the Lov{\'a}sz theta function.
A graph family $\mathcal{G}$ is {\em near optimal colourable} if there is a constant number $c$ such that every graph $G\in\mathcal{G}$ satisfies $\chi(G)\leq\max\{c, \omega(G)\}$, where $\chi(G)$ and $\omega(G)$ are the chromatic number and clique number of $G$, respectively. The near optimal colourable graph families together with the Lov{\'a}sz theta function are useful for the study of the chromatic number problems for hereditary graph families. 
We investigate the near optimal colourability for ($H_1,H_2$)-free graphs. 
Our main result is an almost complete characterization for the near optimal colourability for ($H_1,H_2$)-free graphs
with two exceptional cases, one of which is the celebrated Gy{\'a}rf{\'a}s conjecture.
%To obtain the result, we prove that the family of ($2K_2,P_4\vee K_n$)-free graphs is near optimal colourable for every positive integer $n$ by inductive arguments.%
As an application of our results, we show that the chromatic number problem for ($2K_2,P_4\vee K_n$)-free graphs is polynomial time solvable, which solves an open problem in [K.~K.~Dabrowski and D.~Paulusma. On colouring ($2P_2$, $H$)-free and ($P_5$, $H$)-free graphs. Information Processing Letters, 134:35-41, 2018].
%We give a partial solution to this deciding problem, and there are two subproblems left to solve, one of which is equivalent to the Gy{\'a}rf{\'a}s conjecture. Most of this paper is devoted to solve a subproblem: the family of ($2K_2,P_4\vee K_n$)-free graphs is near optimal colourable for every positive integer $n$.
\end{abstract}

\section{Introduction}\label{sec:intro}

All graphs in this paper are finite and simple. For general graph theory notation we follow~\cite{BM08}. Let $P_n$, $C_n$ and $K_n$ denote the path, cycle and complete graph on $n$ vertices, respectively. 
We denote the complement graph of $G$ by $\overline{G}$. For two vertex disjoint  graphs $G$ and $H$, we write $G+H$ to denote the \emph{disjoint union} of $G$ and $H$, and $G\vee H$
to denote the graph obtained from $G+H$ by adding an edge between every vertex in $G$ and every vertex in $H$. For a positive integer $r$, we use $rG$ to denote the disjoint union of $r$ copies of $G$. A \emph{hole} is an induced cycle on four or more vertices. An \emph{antihole} is the complement of a hole. A hole or antihole is \emph{odd} if it has an odd number of vertices. The graph $K_n-e$ is obtained from $K_n$ by removing an edge. 
The graph $K_1\vee 3K_1$, $(K_2+K_1)\vee K_1$, $K_4-e$, $P_4\vee K_1$, $(K_2+K_1)\vee K_2$ are usually called \emph{claw}, \emph{paw}, \emph{diamond}, \emph{gem} and \emph{HVN}, respectively. 
A \emph{linear forest} is a disjoint union of paths.

%A \emph{linear forest} is a forest whose complements are paths. 

%A \emph{clique} (\emph{stable set}) in a graph $G$ is a set of pairwise adjacent (nonadjacent) vertices in $G$.
We say that a graph $G$ \emph{contains} a graph $H$ if $G$ has an induced subgraph that  is isomorphic to $H$. A graph $G$ is \emph{$H$-free} if $G$ does not contain $H$. For a family $\mathcal{H}$ of graphs, $G$ is \emph{$\mathcal{H}$-free} if $G$ is $H$-free for every $H\in\mathcal{H}$. We write $(H_1,\ldots,H_n)$-free instead of $\{H_1,\ldots,H_n\}$-free. Each graph in $\mathcal{H}$ is called a \emph{forbidden induced subgraph} of the family of $\mathcal{H}$-free graphs. A graph family $\mathcal{G}$ is \emph{hereditary} if $G\in\mathcal{G}$ implies that every induced subgraph of $G$ belongs to $\mathcal{G}$. Clearly, a graph family is hereditary if and only if it is the family of $\mathcal{H}$-free graphs for some graph set $\mathcal{H}$.

A \emph{$q$-colouring} of a graph $G$ is an assignment of colours from $\{1,2,\ldots,q\}$ to each vertex of $G$ such that adjacent vertices receive different colours. We say that a graph $G$ is \emph{$q$-colourable} if $G$ admits a $q$-colouring. The \emph{chromatic number} of a graph $G$, denoted by $\chi(G)$, is the minimum number $q$ for which $G$ is $q$-colourable. The \emph{clique number} of $G$, denoted by $\omega(G)$, is the size of a largest clique in $G$. Clearly,  every graph $G$ satisfies $\chi(G)\geq \omega(G)$. A graph $G$ is \emph{perfect} if $\chi(H)=\omega(H)$ for each induced subgraph $H$ of $G$.
As a generalisation of perfect graphs, Gy{\'a}rf{\'a}s~\cite{Gy87} introduced the $\chi$-bounded graph families. A graph family $\mathcal{G}$ is \emph{$\chi$-bounded} if there is a function $f:\mathbb{N}\rightarrow \mathbb{N}$ such that $\chi(G)\leq f(\omega(G))$ for every $G\in\mathcal{G}$. The function $f$ is called a \emph{$\chi$-binding function} for $\mathcal{G}$.
It follows from a classical result of Erd\H{o}s \cite{Er59} that if a graph $H$ is not a forest, then the family of $H$-free graphs is not $\chi$-bounded. Gy{\'a}rf{\'a}s~\cite{Gy73} conjectured that if a graph $H$ is a forest, then the family of $H$-free graphs is $\chi$-bounded (known as the Gy{\'a}rf{\'a}s conjecture).

%\begin{theorem}[\cite{Er59}]\label{thm:Erdos}
%For any positive integers $k,l\geq 3$, there exists a graph $G$ with $g(G)\geq l$ and $\chi(G)\geq k$.
%\end{theorem}

To determine whether a general graph is $q$-colourable is NP-complete when $q\geq 3$~\cite{Ka72}. This implies that the chromatic number problem for general graphs is NP-complete. However, for some graph families, there exist polynomial time algorithms for the chromatic number problem. For example, Gr{\" o}tschel, Lov{\' a}sz and Schrijver~\cite{GLS84} proved that the chromatic number of a perfect graph can be computed in polynomial time via the Lov{\'a}sz theta function, which is defined as follows:

$$
\begin{aligned}
\vartheta(G):=\max\{\sum\limits_{i,j=1}^n b_{ij}: & B=(b_{ij}) {\rm~ is~ positive~ semidefinite~ with~ trace~ at~ most~ 1,}\\
& {\rm and}~ b_{ij}=0~ {\rm if}~ ij\in E\}.
\end{aligned}
$$
The Lov{\'a}sz theta function can be computed in polynomial time, and satisfies that $\omega(G)\leq\vartheta(\overline{G})\leq\chi(G)$ for any graph $G$~\cite{GLS84}. It follows that $\chi(G)=\vartheta(\overline{G})$ if $G$ is perfect.

In this paper, we define a new notion called near optimal colourable that relates $\chi$-boundedness with computational complexity of graph colouring via the Lov{\'a}sz theta function. 
\begin{definition}[Main Notion]
A graph family $\mathcal{G}$ is said to be near optimal colourable if there is a constant number $c$ such that every graph $G\in \mathcal{G}$ has $\chi(G)\leq \max\{c, \omega(G)\}$. 
\end{definition}

The following theorem shows that near optimal colourability allows one to reduce the chromatic number problem to $k$-colouring for fixed $k$. 
\begin{theorem}\label{thm:reduce_omega}
Suppose that $\mathcal{G}$ is a near optimal colourable family with $c$ being such that every graph $G\in \mathcal{G}$ satisfies
that $\chi(G)\le max\{\omega(G),c\}$. If the $k$-colouring problem for $\mathcal{G}$ is polynomial time solvable for every fixed positive integer $k\le c-1$, then the chromatic number problem for $\mathcal{G}$ is polynomial time solvable.
\end{theorem}
\begin{proof}
Let $G\in\mathcal{G}$. We may calculate $\vartheta(\overline{G})$ in polynomial time~\cite{GLS84}. If $\vartheta(\overline{G})\geq c$, then $\vartheta(\overline{G}) \geq \max\{c,\omega(G)\} \geq \chi(G) \geq \vartheta(\overline{G})$ and so $\chi(G) = \vartheta(\overline{G})$. Now we assume that $\vartheta(\overline{G}) < c$. Then $\omega(G) < c$ and so $\chi(G)\leq c$. So the chromatic number problem is reduced to the $k$-colouring problems for $k\leq c-1$.
% If $\vartheta(\overline{G})>c-1$, since $\chi(G)\leq\max\{c, \omega(G)$\} and $\omega(G)\leq\vartheta(\overline{G})\leq\chi(G)$, we have $\chi(G)=\lceil \vartheta(\overline{G}) \rceil$. If $\vartheta(\overline{G})\leq c-1$, then $\omega(G)\leq c-1$ and $\chi(G)\leq c$, and then the chromatic number problem is reduced to the $k$-colouring problems for $k\leq c-1$.
\end{proof}

Therefore, it is useful to study which hereditary graph classes are near optimal colourable.
It is known that \cite{RS04} the family of $H$-free graph is linearly $\chi$-bounded if and only if $H$ is an induced of $P_4$.
This implies that the family of $H$-free graphs is near optimal colourable if and only if $H$ is an induced subgraph of $P_4$. 
So it is natural to study hereditary graph families defined by two forbidden induced subgraphs.
For example, we (together with Goedgebeur and Merkel) proved that every ($P_6$, diamond)-free graph $G$ satisfies that $\chi(G)\leq \max\{6, \omega(G)$\}, and that the 5-colouring problem for ($P_6$, diamond)-free graphs is polynomial time solvable~\cite{GHJM23}. It then follows from \autoref{thm:reduce_omega} and the known result on the polynomial-time solvibility of 4-colouring $P_6$-free graphs~\cite{CSZ19} that the chromatic number problem for ($P_6$, diamond)-free graphs is polynomial time solvable~\cite{GHJM23}.

%\begin{observation}[\cite{RS04}]\label{obs:linearbound}
%Let $\mathcal{G}$ be the family of $H$-free graphs. If $H$ is an induced subgraph of $P_4$, then $\mathcal{G}$ has a $\chi$-binding function $\omega(G)$, or otherwise there exists no linear $\chi$-binding function for $\mathcal{G}$.
%\end{observation}

\subsection*{Our Contributions}

Our main result is an almost complete characterization for the near optimal colourability for ($H_1,H_2$)-free graphs. 
Let $\mathcal{X}$ be the set of all graphs listed in \ref{item:1}, \ref{item:2}, \ref{item:3}, and \ref{item:4}. 
Let $\mathcal{X}'$ be the set of all graphs listed in \ref{item:1} and \ref{item:2} except the paw.

\begin{enumerate}[(1)]
\setlength{\itemsep}{0pt}

\item $G = P_4\vee K_n$ for $n\geq 1$.\label{item:1}

\item $G = (K_2+K_1)\vee K_n$ for $n\geq 1$.\label{item:2}

\item $G = K_n-e$ for $n\geq 4$.\label{item:3}

\item $G = K_n$ for $n\geq 3$.\label{item:4}
\end{enumerate}

The following is our main result.

\begin{theorem}\label{thm:dividing}
Let $H_1,H_2$ be graphs that are not induced subgraphs of $P_4$. 
%Let $\mathcal{G}$ be the family of $(H_1,H_2)$-free graphs. 
Then the family of $(H_1,H_2)$-free graphs is near optimal colourable only if $H_1$ is a forest and $H_2\in \mathcal{X}$. 
Moreover if $H_2\in \mathcal{X}'$, then the family of $(H_1,H_2)$-free graphs is near optimal colourable if and only if $H_1 = 2K_2$. 
($H_1$ and $H_2$ are interchangeable.)

%Let $H_1,H_2$ be two graphs and $\mathcal{G}$ be the family of $(H_1,H_2)$-free graphs. Then,

%(1) If $H_1$ or $H_2$ is an induced subgraph of $P_4$, then $\mathcal{G}$ is near optimal colourable.

%(2) If neither $H_1$ nor $H_2$ is a forest, then $\mathcal{G}$ is not near optimal colourable.

%(3) If neither $H_1$ nor $H_2$ is a complement of a linear forest, then $\mathcal{G}$ is not near optimal colourable.

%(4) Let $H_1$ be a forest and $H_2$ be a complement of a linear forest, and neither $H_1$ nor $H_2$ is an induced subgraph of $P_4$. Then,

%\todo[inline]{need indent}
%\begin{adjustwidth}{4em}{0em}
%(4.1) If $H_2$ is not isomorphic to any of $P_4\vee K_n$, $(K_2+K_1)\vee K_n$, $K_n-e$, $K_n$ for any $n$, then $\mathcal{G}$ is not near optimal colourable.
%\end{adjustwidth}

%\begin{adjustwidth}{4em}{0em}
%(4.2) If $H_2=P_4\vee K_n$ or $H_2=(K_2+K_1)\vee K_{n+1}$ for some $n\geq 1$, then $\mathcal{G}$ is near optimal colourable if and only if $H_1=2K_2$.
%\end{adjustwidth}
\end{theorem}

Our theorem leaves only two open cases for the near optimal colourability for $(H_1,H_2)$-free graphs: the Gy{\'a}rf{\'a}s conjecture and \autoref{prob:decidecase2} in \autoref{sec:conc} (where we explain why this is the case). Moreover, \autoref{prob:decidecase2} is also related to Gy{\'a}rf{\'a}s conjecture
and so the open cases in some sense are as hard as Gy{\'a}rf{\'a}s conjecture.
%\todo[inline]{Cross reference Problem, Section} 
As an application of our result, we prove the following \autoref{thm:2k2p4kp_poly}. This solves an open problem by Konrad and Paulusma \cite{DP18}.
 
%There are three cases left to solve:
%
%\case{1} $H_1$ is a forest and is not an induced subgraph of $P_4$. $H_2=K_n$ for some $n\geq 3$.
%
%\case{2} $H_1$ is a forest and is not an induced subgraph of $P_4$. $H_2=K_n-e$ for some $n\geq 4$.
%
%\case{3} $H_1$ is a forest and is not an induced subgraph of $P_4$. $H_2$ is a paw.
%
%In \autoref{sec:conc} we show why are these cases left.

\begin{theorem}\label{thm:2k2p4kp_poly}
For every positive integer $n$, the chromatic number problem of ($2K_2,P_4\vee K_n$)-free graphs is polynomial time solvable.
\end{theorem}
\begin{proof}
By \autoref{thm:dividing}, the family of ($2K_2,P_4\vee K_n$)-free graphs is near optimal colourable. 
Since the $k$-colouring problem of $P_5$-free graphs is polynomial time solvable~\cite{HKLSS10}, it follows from \autoref{thm:reduce_omega} that the chromatic number problem of ($2K_2,P_4\vee K_n$)-free graphs is polynomial time solvable.
\end{proof}

The remainder of the paper is organized as follows. In \autoref{sec:pre} we present some preliminaries. In \autoref{sec:main} we reduce \autoref{thm:dividing} to the case that $H_1=2K_2$, $H_2=P_4\vee K_n$. In \autoref{sec:2k2p4kp} we prove the family of $(2K_2,P_4\vee K_n)$-free graphs is near optimal colourable (\autoref{thm:2k2p4kp}). In \autoref{sec:conc} we make some concluding remarks and list some open problems.

\section{Preliminaries}\label{sec:pre}

Let $G=(V,E)$. We say a vertex $u$ is a \emph{neighbour} (\emph{nonneighbour}) of another vertex $v$ in $G$, if $u$ and $v$ are adjacent (nonadjacent). The set of neighbours of a vertex $v$ in $G$ is denoted by $N_G(v)$, and we often write $N(v)$ if the context is clear. We write $\overline{N}(v)$ to denote the set $V\setminus (N(v)\cup \{v\})$. For a set $S\subseteq V$, we write $N_S(v)$ to denote $N(v)\cap S$, and $\overline{N_S}(v)$ to denote $\overline{N}(v)\cap S$. For $S\subseteq V$, let $G[S]$ denote the subgraph of $G$ induced by $S$, and we often write $S$ for $G[S]$ if the context is clear. We say that $S$ induces a graph $H$ if $G[S]$ is isomorphic to $H$. For two sets $S,T\subseteq V$, we say $S$ and $T$ are \emph{complete}(\emph{anti-complete}), if every vertex in $S$ is adjacent (nonadjacent) to every vertex in $T$. (If $S$ or $T$ has only one vertex $v$, we may write $v$ for $\{v\}$.) A vertex set $S\subseteq V$ is a \emph{stable set} if the vertices in $S$ are pairwise nonadjacent.

We will use the following three known results:

\begin{theorem}[\cite{CRST06}]\label{thm:SPGT}
A graph is perfect if and only if it does not contain an odd hole or an odd antihole.
\end{theorem}

\begin{theorem}[\cite{Wa80}]\label{thm:2k2bound}
If $G$ is a $2K_2$-free graph, then $\chi(G)\leq\binom{\omega(G)+1}{2}$.
\end{theorem}

\begin{theorem}[\cite{BRSV19}]\label{thm:2k2gem}
If $G$ is a ($2K_2$,gem)-free graph, then $\chi(G)\leq \max\{3, \omega(G)\}$.
\end{theorem}

\section{The Main  Result}\label{sec:main}

In this section we prove \autoref{thm:dividing}. We start with two lemmas.

%We use the following \autoref{thm:2k2p4kp} to prove \autoref{thm:dividing}. We prove \autoref{thm:2k2p4kp} in \autoref{sec:2k2p4kp}.

\begin{lemma}\label{lem:2k23k1}
Let $H$ be a forest which is not an induced subgraph of $P_4$. Then either $H$ contains a $3K_1$ or $H$ is a $2K_2$, and $\overline{H}$ is not a forest.
\end{lemma}
\begin{proof}
Assume that $H$ is $3K_1$-free, then $H$ is a linear forest and has at most two components. 
Since $P_5$ contains a $3K_1$ and $H$ is not an induced subgraph of $P_4$, $H$ has two components.
Then each component of $H$ has at most two vertices. This implies that $H$ is a $2K_2$ since $H$ is not an induced subgraph of $P_4$. So $H$ contains a $3K_1$ or is a $2K_2$, which implies that $\overline{H}$ contains a $C_3$ or is a $C_4$.
\end{proof}

\begin{lemma}\label{lem:omegageq}
Let $\mathcal{G}$ be a $\chi$-bounded graph family. Then $\mathcal{G}$ is near optimal colourable if and only if there is a constant number $g$ such that every graph $G\in \mathcal{G}$ with $\omega(G)\geq g$ has $\chi(G)=\omega(G)$.
\end{lemma}

\begin{proof}

The necessity is obvious. 
Suppose that every graph $G\in \mathcal{G}$ with $\omega(G)\geq g$ has $\chi(G)=\omega(G)$. Since $\mathcal{G}$ is $\chi$-bounded, we may assume that every graph $G\in \mathcal{G}$ with $\omega(G) < g$ has $\chi(G)\leq c$. So every graph $G\in \mathcal{G}$ has $\chi(G)\leq \max\{c,\omega(G)\}$.
%Let $f$ be a $\chi$-binding function of $\mathcal{G}$.
%Suppose that every graph $G\in \mathcal{G}$ with $\omega(G)\geq g$ has $\chi(G)=\omega(G)$. Let $H\in \mathcal{G}$. If $\omega(H)\geq g$, then $\chi(H)=\omega(H)$. If $\omega(H)\leq g-1$, then $\chi(H)\leq f(g-1)$. So $\chi(H)\leq\max\{f(g-1), \omega(H)\}$.
\end{proof}

Now we define two families of graphs (see \autoref{fig:xnyn}). For every positive integer $n$, let $X_n=C_5\vee K_n$, then $\chi(X_n)=n+3$ and $\omega(X_n)=n+2$. The graph $Y_n$ is obtained from a $C_5$ by blowing up two nonadjacent vertices to $K_n$, then $\chi(Y_n)=n+2$ and $\omega(Y_n)=n+1$. (To \emph{blow up} a vertex $v$ in $G$ to a graph $H$, is to remove $v$ and add a graph $H$ and then make $H$ complete to $N_G(v)$ and anti-complete to $\overline{N_G}(v)$.)  Note that every $X_n$ is ($2K_2,3K_1,C_4$)-free, and every $Y_n$ is (gem, HVN, $3K_1, C_4$)-free. (The constructions of $X_n$ and $Y_n$ can also be found in \cite{BHPT93} and \cite{Xu23}.)

\begin{figure}[h!]
\centering
\begin{tikzpicture}[scale=0.5]
\tikzstyle{vertex}=[circle, draw, fill=white, inner sep=1pt, minimum size=5pt]
	\node at (0,-2.5) {$X_n$};
    \node[vertex](1) at (0,2) {};
    \node[vertex](2) at (-{2*cos(18)},{2*sin(18}) {};
    \node[vertex](3) at (-{2*sin(36)},-{2*cos(36}) {};
    \node[vertex](4) at ({2*sin(36)},-{2*cos(36}) {};
    \node[vertex](5) at ({2*cos(18)},{2*sin(18}) {};
    \node[vertex](6) at (0,0) {$K_n$};

	\foreach \from/\to in {1/2,2/3,3/4,4/5,1/5,1/6,2/6,3/6,4/6,5/6}
		\draw (\from) -- (\to);

\end{tikzpicture}
\hspace{10mm}
\begin{tikzpicture}[scale=0.5]
\tikzstyle{vertex}=[circle, draw, fill=white, inner sep=1pt, minimum size=5pt]
	\node at (0,-2.5) {$Y_n$};
    \node[vertex](1) at (0,2) {};
    \node[vertex](2) at (-{2*cos(18)},{2*sin(18}) {$K_n$};
    \node[vertex](3) at (-{2*sin(36)},-{2*cos(36}) {};
    \node[vertex](4) at ({2*sin(36)},-{2*cos(36}) {};
    \node[vertex](5) at ({2*cos(18)},{2*sin(18}) {$K_n$};

	\foreach \from/\to in {1/2,2/3,3/4,4/5,1/5}
		\draw (\from) -- (\to);
\end{tikzpicture}
\caption{}
\label{fig:xnyn}
\end{figure}

To prove \autoref{thm:dividing}, we need to handle a particular graph families, namely ($2K_2,P_4\vee K_n$)-free graphs.

\begin{theorem}\label{thm:2k2p4kp}
For every positive integer $n$, the family of ($2K_2,P_4\vee K_n$)-free graphs is near optimal colourable.
\end{theorem}

We shall prove \autoref{thm:2k2p4kp} in \autoref{sec:2k2p4kp}. Now we use \autoref{thm:2k2p4kp} to prove \autoref{thm:dividing}.

\begin{proof}[Proof of \autoref{thm:dividing}]
%~\\

%Note that \autoref{thm:dividing}(1) follows from \autoref{thm:SPGT}, and \autoref{thm:dividing}(2)follows from the result of Erd\H{o}s \cite{Er59}.
Let  $\mathcal{G}$ denote the family of $(H_1,H_2)$-free graphs. Assume that $\mathcal{G}$ is near optimal colourable. By the result of Erd\H{o}s \cite{Er59}, if neither $H_1$ nor $H_2$ is a forest, then $\mathcal{G}$ is not $\chi$-bounded and so not near optimal colourable. By symmetry let $H_1$ be a forest.

Assume that $H_2$ is not a complement of a linear forest. By \autoref{lem:2k23k1}, $H_1$ is not a complement of a linear forest, either. By \autoref{lem:omegageq}, there exists a constant number $g$ such that every graph $G\in\mathcal{G}$ with $\omega(G)\geq g$ has $\chi(G)=\omega(G)$. Since neither $H_1$ nor $H_2$ is a complement of a linear forest, each of $\overline{H_1}$ and $\overline{H_2}$ contains a cycle or a claw. Assume that the longest induced cycle in $\overline{H_1}$ and $\overline{H_2}$ has $m$ vertices, if it exists. If neither $\overline{H_1}$ nor $\overline{H_2}$ contains a cycle, then let $m$ be $0$. We consider an odd antihole $L=\overline{C_{2n+1}}$ with $n\geq\max\{g, \frac{m}{2}\}$. Then $\chi(L)=n+1$ and $\omega(L)=n$. Since $\overline L=C_{2n+1}$ does not contain a cycle of size $m$ or less (if $m>0$) or a claw, $\overline L$ is ($\overline{H_1}, \overline{H_2}$)-free. So $L\in\mathcal{G}$, which contradicts our assumption that every graph $G\in\mathcal{G}$ with $\omega(G)\geq g$ has $\chi(G)=\omega(G)$. So $H_2$ is a complement of a linear forest.

By \autoref{lem:2k23k1}, $H_1$ contains either a $3K_1$ or a $2K_2$. If $H_2$ contains a $C_4$, then $\mathcal{G}$ contains $X_n$ for arbitrarily large $n$, which contradicts that $\mathcal{G}$ is near optimal colourable. So $H_2$ must be $C_4$-free, and then $\overline{H_2}$ is a $2K_2$-free linear forest. So $\overline{H_2}$ is an induced subgraph of $P_4+nK_1$ for some $n$. It is easy to check that $\mathcal{X}$ are all graphs satisfying the conditions for $H_2$. Now we assume that $H_2$ is $C_4$-free and contains a gem or an HVN, that is, $H_2\in \mathcal{X}'$. If $H_1$ contains a $3K_1$, then $\mathcal{G}$ contains $Y_n$ for arbitrarily large $n$, which contradicts that $\mathcal{G}$ is near optimal colourable. By \autoref{lem:2k23k1}, $H_1$ can only be $2K_2$. This proves the necessity. The sufficiency follows from \autoref{thm:2k2p4kp} (note that every ($2K_2,(K_2+K_1)\vee K_n$)-free graph is ($2K_2,P_4\vee K_n$)-free).
%Assume that $\mathcal{G}$ is near optimal colourable. First we assume that $H_2$ contains a $C_4$, then $H_1$ is ($2K_2,3K_1$)-free, or else $\mathcal{G}$ contains $X_n$ for arbitrarily large $n$. By \autoref{lem:2k23k1}, there is no such an $H_1$ that meets the conditions. So $H_2$ must be $C_4$-free, and then $\overline{H_2}$ is a $2K_2$-free linear forest. So $\overline{H_2}$ is an induced subgraph of $P_4+nK_1$ for some $n$. It is easy to check that $\mathcal{X}$ are all graphs satisfying the conditions for $H_2$. Now we assume that $H_2$ is $C_4$-free and contains a gem or an HVN, that is, $H_2\in \mathcal{X}'$. Then $H_1$ must be $3K_1$-free, or else $\mathcal{G}$ contains $Y_n$ for arbitrarily large $n$. By \autoref{lem:2k23k1}, $H_1$ can only be $2K_2$. This proves the necessity. The sufficiency follows from \autoref{thm:2k2p4kp}. Note that the family of ($2K_2,(K_2+K_1)\vee K_n$)-free graphs is a subset of the family of ($2K_2,P_4\vee K_n$)-free graphs.
% This proves (4.1). Then we assume that $H_2$ is $C_4$-free and contains a gem or an HVN, that is, $H_2=P_4\vee K_n$ or $H_2=(K_2+K_1)\vee K_{n+1}$ for some $n\geq 1$. Then $H_1$ must be $3K_1$-free, or else $\mathcal{G}$ contains $Y_n$ for arbitrarily large $n$. By \autoref{lem:2k23k1}, $H_1$ can only be $2K_2$. This proves the necessity of (4.2). The sufficiency of (4.2) follows from \autoref{thm:2k2p4kp}. (Note that the family of ($2K_2,(K_2+K_1)\vee K_n$)-free graphs is a subset of the family of ($2K_2,P_4\vee K_n$)-free graphs.)
\end{proof}

\section{($2K_2,P_4\vee K_n$)-free Graphs}\label{sec:2k2p4kp}

In this section we prove \autoref{thm:2k2p4kp}. Brause, Randerath, Schiermeyer, and Vumar~\cite{BRSV19} proved that the family of ($2K_2, P_4\vee K_1$)-free graphs is near optimal colourable (\autoref{thm:2k2gem}). Our \autoref{thm:2k2p4kp} is a generalization of their result.

We start with a simple proposition.

%Since the family of $2K_2$-free graphs is $\chi$-bounded by \autoref{thm:2k2bound}, we use \autoref{lem:omegageq} and divide \autoref{thm:2k2p4kp} to two subproblems, as the following:

\begin{lemma}\label{thm:big_antihole}
For every positive integer $n$, there exists a constant number $c$ such that every ($2K_2, P_4\vee K_n$)-free graph $G$ which contains an antihole on 6 or more vertices has $\chi(G)\leq c$.
\end{lemma}

\begin{proof}
Let $G=(V,E)$ be ($2K_2, P_4\vee K_n$)-free. Let $Q=\{v_1,v_2,\ldots,v_r\}(r\geq6)$ be an antihole in $G$ with $v_iv_{i+1}\notin E$ for $i=1,2,\ldots,r$, with all indices modulo $r$. Since $\overline{C_{2n+5}}$ contains a $P_4\vee K_n$, we have $r\leq 2n+4$. Then $\omega(Q)=\lfloor\frac{r}{2}\rfloor\leq n+2$.

For each subset $S$ of $Q$, we denote $N_S=\{u\in V\setminus Q: N_Q(u)=S\}$. Then every vertex in $V\setminus Q$ belongs to exactly one of these $2^r$ sets. For any subset $S$ of $Q$ we consider the following two cases:

{\bf Case 1.} $Q\setminus S$ contains a $K_2$. Then $N_S$ is a stable set since $G$ is $2K_2$-free.

{\bf Case 2.} $Q\setminus S$ is a stable set. Then there exist an $i$ such that $Q\setminus S\subseteq\{v_i,v_{i+1}\}$, and $\{v_{i+2},v_{i+3},v_{i+4},v_{i+5}\}\subseteq S$ induces a $P_4$. So $\omega(N_S)\leq n-1$ since $G$ is $P_4\vee K_n$-free.

So for every $S$, $\omega(N_S)\leq\max\{1,n-1\}\leq n$, then $\omega(G)\leq 2^r\cdot n+\omega(Q)\leq 2^{2n+4}\cdot n+n+2$ and $\chi(G)\leq\binom{2^{2n+4}\cdot n+n+3}{2}$ by \autoref{thm:2k2bound}.
\end{proof}

To prove \autoref{thm:2k2p4kp}, it remains to prove the following theorem on $(2K_2, P_4\vee K_n)$-free graphs with no antihole of length 6 or more. 

\begin{theorem}\label{thm:c5}
For every positive integer $n$, the family of ($2K_2, P_4\vee K_n$)-free graph $G$ which contains no antiholes on 6 or more vertices is near optimal colourable.
%For every positive integer $n$, there exists a constant number $g$ such that every ($2K_2, P_4\vee K_n$)-free graph $G$ which contains no antiholes on 6 or more vertices with $\omega(G)\geq g$ has $\chi(G)=\omega(G)$.
\end{theorem}

%Then \autoref{thm:2k2p4kp} follows from \autoref{thm:big_antihole} and \autoref{thm:c5}.

\begin{proof}
By \autoref{thm:2k2bound} and \autoref{lem:omegageq}, it suffices to prove that for every positive integer $n$, there exists a constant number $g$ such that every ($2K_2, P_4\vee K_n$)-free graph $G$ which contains no antiholes on 6 or more vertices with $\omega(G)\geq g$ has $\chi(G)=\omega(G)$. We prove by induction on $n$. By \autoref{thm:2k2gem}, this theorem is true for $n=1$. Now we assume that $n\geq 2$, and every ($2K_2, P_4\vee K_{n-1}$)-free graph $G$ which contains no antiholes on 6 or more vertices with $\omega(G)\geq g_{n-1}$ has $\chi(G)=\omega(G)$.

Let $G=(V,E)$ be a ($2K_2, P_4\vee K_n$)-free graph which contains no antiholes on 6 or more vertices. We assume that $\omega(G)\geq\max\{g_{n-1}, 4n+2\}+10n$, and we prove that $\chi(G)=\omega(G)$ in the following.

We say two nonadjacent vertices $u,v$ are \emph{comparable} if $N(u)\subseteq N(v)$ or $N(v)\subseteq N(u)$. If $u_1,u_2\in V$ are comparable with $N(u_1)\subseteq N(u_2)$, then $\chi(G-u_1)=\chi(G)$ and $\omega(G-u_1)=\omega(G)$. So we may assume that $G$ has no pairs of comparable vertices.

If $G$ is $C_5$-free, then $G$ is perfect by \autoref{thm:SPGT}. So in the following we assume that $Q=\{v_1,v_2,v_3,v_4,v_5\}$ induces a $C_5$ in $G$ with $v_iv_{i+1}\in E$ for $i=1,2,\ldots,5$, with all indices modulo 5. We define some vertex sets:

$A_i=\{u\in V\setminus Q: N_Q(u)=\{v_{i-1},v_{i+1}\}\}$.

$B_i=\{u\in V\setminus Q: N_Q(u)=\{v_i,v_{i-2},v_{i+2}\}\}$.

$D_i=\{u\in V\setminus Q: N_Q(u)=Q\setminus\{v_i\}\}$.

$F=\{u\in V\setminus Q: N_Q(u)=Q\}$.

$Z=\{u\in V\setminus Q: N_Q(u)=\emptyset\}$.

$A=\bigcup_{i=1}^5A_i$, $B=\bigcup_{i=1}^5B_i$, $D=\bigcup_{i=1}^5D_i$.

$S_i=A_i\cup\{v_i\}$, $S=\bigcup_{i=1}^5S_i$.

We prove some properties of these sets.

\begin{claim}\label{clm:whole}
$V=Q\cup A\cup B\cup D\cup F\cup Z$.
\end{claim}
\begin{proof}
Let $u\in V\setminus Q$. If $u$ has no neighbours in $Q$, then $u\in Z$. By symmetry assume that $uv_1\in E$. Since $\{u,v_1,v_3,v_4\}$ cannot induce a $2K_2$, $u$ must be adjacent to at least one of $v_3$ and $v_4$. By symmetry let $uv_3\in E$. If $uv_2\notin E$, then depending on $uv_4$ and $uv_5$, there are four cases that $u$ belongs to $A_2, B_1, B_3$ or $D_2$. If $uv_2\in E$, then sinse $\{u,v_2,v_4,v_5\}$ cannot induce a $2K_2$, $u$ must be adjacent to at least one of $v_4$ and $v_5$. Then there are three cases that $u$ belongs to $D_4,D_5$ or $F$.
\end{proof}

\begin{claim}\label{clm:astable}
Each $S_i\cup Z$ is a stable set. As a corollary, $Z$ and $S$ are anti-complete.
\end{claim}
\begin{proof}
If $u_1,u_2\in S_i\cup Z$ are adjacent, then $\{u_1,u_2,v_{i-2},v_{i+2}\}$ induces a $2K_2$.
\end{proof}

\begin{claim}\label{clm:dfomega}
Each $D_i\cup F$ has clique number at most $n-1$.
\end{claim}
\begin{proof}
If $K\subseteq D_i\cup F$ is a clique on $n$ vertices, then $K\cup\{v_{i+1},v_{i+2},v_{i-2},v_{i-1}\}$ induces a $P_4\vee K_n$.
\end{proof}

\begin{claim}\label{clm:bb+1}
$B_i$ and $B_{i-1}\cup B_{i+1}\cup A_i$ are complete.
\end{claim}
\begin{proof}
If $u_1\in B_i$ and $u_2\in B_{i+1}\cup A_i$ are nonadjacent, then $\{u_1,v_{i+2},u_2,v_{i-1}\}$ induces a $2K_2$.
\end{proof}

\begin{claim}\label{clm:nonneighbourhoodstable}
Let $u\in A_{i+2}\cup A_{i-2}\cup B_{i+2}\cup B_{i-2}\cup D\cup F$, then $\overline{N_{B_i}}(u)$ is a stable set.
\end{claim}
\begin{proof}
It is clear that $u$ is adjacent to $v_{i+1}$ or $v_{i-1}$. By symmetry let $uv_{i+1}\in E$. Assume that $u_1,u_2\in \overline{N_{B_i}}(u)$ are adjacent, then $\{u,v_{i+1},u_1,u_2\}$ induces a $2K_2$.
\end{proof}

\begin{claim}\label{clm:2b2dempty}
If $\omega(B_i)\geq n+1$, then $B_{i+2}\cup B_{i-2}\cup D_{i+2}\cup D_{i-2}=\emptyset$.
\end{claim}
\begin{proof}
By symmetry suppose that $u\in B_{i+2}\cup D_{i-2}$. By \autoref{clm:nonneighbourhoodstable}, there is a clique $K$ on $n$ vertices in $N_{B_i}(u)$, then $\{v_i,u,v_{i+2},v_{i-2}\}\cup K$ induces a $P_4\vee K_n$.
\end{proof}

\begin{claim}\label{clm:neighbourhoodclique}
Let $u\in A_{i+1}\cup A_{i-1}$, then $\omega(N_{B_i}(u))\leq n-1$.
\end{claim}
\begin{proof}
If $K\subseteq N_{B_i}(u)$ is a clique on $n$ vertices, then $\{v_i,u,v_{i+2},v_{i-2}\}\cup K$ induces a $P_4\vee K_n$.
\end{proof}

If every $B_i$ has $\omega(B_i)\leq n$, then $\omega(G)\leq\omega(S)+\omega(B)+\omega(D)+\omega(F)+\omega(Z)\leq 5+5n+5(n-1)+(n-1)+1=11n$, which contradicts our assumption that $\omega(G)\geq 14n+2$. By symmetry we may assume that $\omega(B_1)\geq n+1$. By \autoref{clm:2b2dempty}, $B_3\cup B_4\cup D_3\cup D_4$ is empty. Moreover, at most one of $B_2$ and $B_5$ has clique number at least $n+1$. By symmetry we assume that $\omega(B_5)\leq n$. Then we discuss two cases based on whether $\omega(B_2)\geq n+1$.

\case{1} $\omega(B_2)\geq n+1$.

By \autoref{clm:2b2dempty}, $B=B_1\cup B_2$ and $D=D_1\cup D_2$. So $B\cup D\cup F\subseteq N(v_4)$. By \autoref{clm:astable}, for every $u\in Z$, $N(u)\subseteq N(v_4)$. So $Z=\emptyset$ since $G$ has no pairs of comparable vertices.

\begin{claim}\label{clm:case1_bd}
$D_2\cup F$ and $B_2$ are complete. (resp. $D_1\cup F$ and $B_1$ are complete.)
\end{claim}
\begin{proof}
Suppose that $u_1\in D_2\cup F$ and $u_2\in B_2$ are nonadjacent. By \autoref{clm:nonneighbourhoodstable}, there is a clique $K\subseteq N_{B_1}(u_1)$ on $n$ vertices. By \autoref{clm:bb+1}, $u_2$ and $K$ are complete, then $\{v_1,u_1,v_4,u_2\}\cup K$ induces a $P_4\vee K_n$.
\end{proof}

\begin{claim}\label{clm:case1_dd}
$D_1,D_2,F$ are pairwise complete.
\end{claim}
\begin{proof}
By symmetry suppose that $u_1\in D_1$ and $u_2\in D_2\cup F$ are nonadjacent. By \autoref{clm:nonneighbourhoodstable}, there is a vertex $u_3\subseteq B_2$ such that $u_1u_3,u_2u_3\in E$, and there is a clique $K\subseteq N_{B_1}(u_2)$ on $n$ vertices. By \autoref{clm:case1_bd} and \autoref{clm:bb+1}, $\{u_1,u_3\}$ and $K$ are complete, then $\{v_1,u_2,u_3,u_1\}\cup K$ induces a $P_4\vee K_n$.
\end{proof}

\begin{claim}\label{clm:case1_bdssomega}
$\omega(B_1\cup D_2\cup S_2\cup S_5)=\omega(B_1\cup D_2)$. (resp. $\omega(B_2\cup D_1\cup S_1\cup S_3)=\omega(B_2\cup D_1)$.)
\end{claim}
\begin{proof}
Suppose that $K\subseteq B_1\cup D_2\cup S_2\cup S_5$ is a clique on $\omega(B_1\cup D_2)+1$ vertices. If $|K\cap(B_1\cup D_2)|\leq\omega(B_1\cup D_2)-2$, then $|K|\leq\omega(S_2\cup S_5)+|K\cap(B_1\cup D_2)|\leq\omega(B_1\cup D_2)$, a contradiction. So $|K\cap(B_1\cup D_2)|\geq\omega(B_1\cup D_2)-1\geq n$. Let $L\subseteq K\cap(B_1\cup D_2)$ be a clique on $n$ vertices. Since $|K|>\omega(B_1\cup D_2)$, there is a vertex $u\in K\cap(S_2\cup S_5)$, then $\{v_1,u,v_3,v_4\}\cup L$ induces a $P_4\vee K_n$.
\end{proof}

\begin{claim}\label{clm:case1_homega}
$\omega(G-S_4)=\omega(B_1\cup D_2)+\omega(B_2\cup D_1)+\omega(F)=\omega(G)-1$.
\end{claim}
\begin{proof}
Since $B_1\cup D_2$, $B_2\cup D_1$ and $F$ are pairwise complete by \autoref{clm:case1_bd} and \autoref{clm:case1_dd}, we have that $\omega(G-S_4)\geq\omega(B_1\cup D_2)+\omega(B_2\cup D_1)+\omega(F)$. By \autoref{clm:case1_bdssomega},
$$
\begin{aligned}
\omega(G-S_4) & \leq \omega(B_1\cup D_2\cup S_2\cup S_5)+\omega(B_2\cup D_1\cup S_1\cup S_3)+\omega(F) \\
& =\omega(B_1\cup D_2)+\omega(B_2\cup D_1)+\omega(F)
\end{aligned}
$$
So $\omega(G-S_4)=\omega(B_1\cup D_2)+\omega(B_2\cup D_1)+\omega(F)$. Therefore, there is a clique $L\subseteq B_1\cup B_2\cup D_1\cup D_2\cup F$ on $\omega(G-S_4)$ vertices.

Since $S_4$ is a stable set, we have that $\omega(G-S_4)\geq\omega(G)-1$. If $\omega(G-S_4)=\omega(G)$, then $L\cup\{v_4\}$ is a clique on $\omega(G)+1$ vertices since $v_4$ is complete to $B_1\cup B_2\cup D_1\cup D_2\cup F$. This proves that $\omega(G-S_4)=\omega(G)-1$.
\end{proof}

\begin{claim}\label{clm:case1_partsc5free}
Each of $B_1\cup D_2\cup S_2\cup S_5$, $B_2\cup D_1\cup S_1\cup S_3$ and $F$ is $C_5$-free.
\end{claim}
\begin{proof}
Since $F$ and $B_1$ are complete, and $\omega(B_1)\geq n$, we have that $F$ is $P_4$-free.

Suppose that $R\subseteq B_1\cup D_2\cup S_2\cup S_5$ induces a $C_5$. Since $S_5$ is a stable set, $R$ has at most two vertices in $S_5$. Choose an induced $P=P_4\subseteq R$ such that $|P\cap S_5|\leq 1$. Note that $B_1\cup D_2\cup S_2$ and $B_2$ are complete. If $P\cap S_5=\emptyset$, then $P\cup B_2$ contains a $P_4\vee K_n$. If $P\cap S_5=\{u\}$, then by \autoref{clm:nonneighbourhoodstable}, there is a clique $K\subseteq N_{B_2}(u)$ on $n$ vertices.  Then $P\cup K$ induces a $P_4\vee K_n$.
\end{proof}

Then,
$$
\begin{aligned}
\chi(G) & \leq\chi(B_1\cup D_2\cup S_2\cup S_5)+\chi(B_2\cup D_1\cup S_1\cup S_3)+\chi(F)+\chi(S_4) \\
& =\omega(B_1\cup D_2\cup S_2\cup S_5)+\omega(B_2\cup D_1\cup S_1\cup S_3)+\omega(F)+1 \\
& =\omega(B_1\cup D_2)+\omega(B_2\cup D_1)+\omega(F)+1\\
& =\omega(G).
\end{aligned}
$$

\case{2} $\omega(B_2)\leq n$.

By \autoref{clm:2b2dempty}, $B=B_1\cup B_2\cup B_5$ and $D=D_1\cup D_2\cup D_5$. Since $\omega(G-B_1)\leq\omega(S)+\omega(B_2\cup B_5)+\omega(D)+\omega(F)+\omega(Z)\leq 5+2n+3(n-1)+(n-1)+1=6n+2$, we have that $\omega(B_1)\geq\omega(G)-(6n+2)\geq\max\{g_{n-1}, 4n+2\}+(4n-2)$.

\begin{claim}\label{clm:case2_zneighbourhoodclique}
Let $u\in Z$, then $\omega(N_{B_1}(u))\leq n-1$.
\end{claim}
\begin{proof}
If $u$ and $B_2$ are anti-complete, then by \autoref{clm:astable}, $N(u)\subseteq B_1\cup B_5\cup D\cup F\subseteq N(v_3)$, which contradicts our assumption that $G$ has no pairs of comparable vertices. So there is a vertex $b$ in $N_{B_2}(u)$. Suppose that $K\subseteq N_{B_1}(u)$ is a clique on $n$ vertices, then $\{u,b,v_4,v_3\}\cup K$ induces a $P_4\vee K_n$.
\end{proof}

Let $H=B\cup D\cup F\cup S_1\cup S_3\cup S_4=V\setminus(S_2\cup S_5\cup Z)$. By \autoref{clm:bb+1} and \autoref{clm:nonneighbourhoodstable}, $S_1\cup B_2\cup B_5$ and $B_1$ are complete, and for every $u\in S_3\cup S_4\cup D\cup F$, $\overline{N_{B_1}}(u)$ is a stable set. Let $K$ be a maximum clique of $G$. Since $\omega(G-B_1)\leq 6n+2$, we have $|K\cap B_1|\geq n$. By \autoref{clm:neighbourhoodclique} and \autoref{clm:case2_zneighbourhoodclique}, $K\cap(S_2\cup S_5\cup Z)=\emptyset$. So $\omega(H)=\omega(G)$.

\begin{claim}\label{clm:case2_diffnonneighbour}
Let $K\subseteq B_1$ be a clique on $n+2$ vertices. If there are $a_1,a_2\in S_3\cup S_4\cup D\cup F$ and $b_1,b_2\in K$ such that $a_1b_1,a_2b_2\notin E$, then $a_1a_2\in E$.
\end{claim}
\begin{proof}
By \autoref{clm:nonneighbourhoodstable}, $a_1b_2,a_2b_1\in E$, and $\{a_1,a_2\}$ and $K\setminus\{b_1,b_2\}$ are complete. Suppose that $a_1a_2\notin E$, then $\{a_1,b_2,b_1,a_2\}$ induces a $P_4$, and then $\{a_1,b_2,b_1,a_2\}\cup(K\setminus\{b_1,b_2\})$ induces a $P_4\vee K_n$.
\end{proof}

\begin{claim}\label{clm:case2_omegaofcomplete}
Let $K\subseteq B_1$ be a clique on $m$ vertices with $m\geq 4n-1$, then there are at least $m-(4n-2)$ vertices in $K$ which are complete to $H\setminus B_1$.
\end{claim}
\begin{proof}
Suppose that there are $b_1,b_2,\ldots,b_{4n-1}\in K$, such that for every $i\in\{1,2,\ldots,4n-1\}$, $b_i$ and $H\setminus B_1$ are not complete. By \autoref{clm:nonneighbourhoodstable}, there are $a_1,a_2,\ldots,a_{4n-1}\in S_3\cup S_4\cup D\cup F$, such that $a_ib_i\notin E$ for every $i\in\{1,2,\ldots,4n-1\}$. By \autoref{clm:case2_diffnonneighbour}, $\{a_1,a_2,\ldots,a_{4n-1}\}$ induces a $K_{4n-1}$. But $\omega(S_3\cup S_4\cup D\cup F)\leq 2+3(n-1)+(n-1)=4n-2$, a contradiction.
\end{proof}

\begin{claim}\label{clm:case2_htog}
If $\chi(H)=\omega(H)$, then $\chi(G)=\omega(G)$.
\end{claim}
\begin{proof}
Since $\omega(B_1)\geq \max\{g_{n-1}, 4n+2\}+(4n-2)$, there is a clique $K\subseteq B_1$ on $\max\{g_{n-1}, 4n+2\}$ vertices which is complete to $H\setminus B_1$ by \autoref{clm:case2_omegaofcomplete}. Suppose that $H$ is already coloured by $\omega(H)$ colours. Let $L=\{u\in H: u$ has the same colour with some vertex in $K$.\} and $M=H\setminus L$. Then $\chi(L)\leq\max\{g_{n-1}, 4n+2\}$, and $\chi(M)\leq\omega(H)-\max\{g_{n-1}, 4n+2\}$. Since $K\subseteq L$, we have $\omega(L)\geq|K|=\max\{g_{n-1}, 4n+2\}$, and so $\chi(L)=\omega(L)=\max\{g_{n-1}, 4n+2\}$. Since $K$ and $H\setminus B_1$ are complete, we have $L\subseteq B_1$.

By \autoref{clm:neighbourhoodclique} and \autoref{clm:case2_zneighbourhoodclique}, $\omega(L\cup S_2\cup S_5\cup Z)=\omega(L)$. Suppose that $P\subseteq L\cup S_2\cup S_5\cup Z$ induces a $P_4\vee K_{n-1}$. If $P\subseteq L\cup S_2\cup S_5$, then $P\subseteq N(v_1)$, and so $P\cup\{v_1\}$ induces a $P_4\vee K_n$. So there is a vertex $u\in P\cap Z$. Note that $\omega(N_P(u))=n$. By \autoref{clm:astable}, $\omega(N_{P\cap B_1}(u))=n$. But $\omega(N_{B_1}(u))\leq n-1$, a contradiction. So $L\cup S_2\cup S_5\cup Z$ is $P_4\vee K_{n-1}$-free. Since $\omega(L)\geq g_{n-1}$, we have $\chi(L\cup S_2\cup S_5\cup Z)=\omega(L)$. Then,
$$
\begin{aligned}
\chi(G) & \leq\chi(L\cup S_2\cup S_5\cup Z)+\chi(M) \\
& \leq\max\{g_{n-1}, 4n+2\}+(\omega(H)-\max\{g_{n-1}, 4n+2\}) \\
& =\omega(G).
\end{aligned}
$$
\end{proof}

\begin{claim}\label{clm:case2_3p+3}
Suppose that $L_1,L_2$ are disjoint cliques in $G$, such that $|L_1|=|L_2|=\omega(L_1\cup L_2)\geq 3n+3$. Then every vertex in $L_1$ ($L_2$) has exactly one nonneighbour in $L_2$ ($L_1$).
\end{claim}
\begin{proof}
Suppose that $u\in L_1$ and $L_2$ are complete, then $\omega(L_1\cup L_2)\geq|L_2|+1$, a contradiction.

\begin{comment}
There is a vertex $s_1\in L_1$, such that for any $t_1,t_2\in L_1\setminus\{s_1\}$, $|N_{L_2}(t_1)\cap N_{L_2}(t_2)|\geq n$. Moreover if there is a vertex $t_1\in L_1\setminus\{s_1\}$ such that $|N_{L_2}(t_1)\cap N_{L_2}(s_1)|\leq n-1$, then $|N_{L_2}(s_1)|\leq \lfloor\frac{|L_1|+n+1}{2}\rfloor$. (By symmetry, there is a vertex $s_2\in L_2$, such that for any $t_1,t_2\in L_2\setminus\{s_2\}$, $|N_{L_1}(t_1)\cap N_{L_1}(t_2)|\geq n$.)
\end{comment}

Now we prove that: 

There is a vertex $s_1\in L_1$, such that for any $t_1,t_2\in L_1\setminus\{s_1\}$, $|N_{L_2}(t_1)\cap N_{L_2}(t_2)|\geq n$. Moreover if there is a vertex $t_1\in L_1\setminus\{s_1\}$ such that $|N_{L_2}(t_1)\cap N_{L_2}(s_1)|\leq n-1$, then $|N_{L_2}(s_1)|\leq \lfloor\frac{|L_1|+n+1}{2}\rfloor$. (By symmetry, there is a vertex $s_2\in L_2$, such that for any $t_1,t_2\in L_2\setminus\{s_2\}$, $|N_{L_1}(t_1)\cap N_{L_1}(t_2)|\geq n$.)\hfill{(1)}

If for any $r_1,r_2\in L_1$, $|N_{L_2}(r_1)\cap N_{L_2}(r_2)|\geq n$, then we are done. Suppose that $a_1,a_2\in L_1$ such that $|N_{L_2}(r_1)\cap N_{L_2}(r_2)|\leq n-1$, then,
$$
\begin{aligned}
|N_{L_2}(r_1)|+|N_{L_2}(r_2)| & =|N_{L_2}(r_1)\cup N_{L_2}(r_2)| + |N_{L_2}(r_1)\cap N_{L_2}(r_2)| \\
& \leq |L_2| + |N_{L_2}(r_1)\cap N_{L_2}(r_2)| \\
& \leq |L_1|+n-1.
\end{aligned}
$$
By symmetry we may assume that $|N_{L_2}(r_1)|\leq \lfloor\frac{|L_1|+n+1}{2}\rfloor$. Then $|\overline{N_{L_2}}(r_1)|\geq \lceil\frac{|L_1|+n+1}{2}\rceil\geq n+2$. Since $G$ is $2K_2$-free, every vertex in $L_1\setminus{r_1}$ has at most one nonneighbour in $\overline{N_{L_2}}(r_1)$. So for any $r_3,r_4\in L_1\setminus\{r_1\}$, we have $|N_{\overline{N_{L_2}}(r_1)}(r_3)\cap N_{\overline{N_{L_2}}(r_1)}(r_4)|\geq |\overline{N_{L_2}}(r_1)|-2\geq n$. Let $r_1$ be the $s_1$ we are finding, then we complete the proof of (1).

Now suppose that $u_1\in L_1$ has nonneighbours $u_3,u_4\in L_2$. Since $G$ is $2K_2$-free, every vertex in $L_1\setminus\{u_1\}$ is adjacent to $u_3$ or $u_4$. If $L_1\setminus\{u_1\}$ and $\{u_3,u_4\}$ are complete, then $\omega(L_1\cup L_2)\geq |L_1|+1$. So there is a vertex $u_2\in L_1$ which is adjacent to exactly one of $u_3$ and $u_4$, say $u_3$, then $\{u_1,u_2,u_3,u_4\}$ induces a $P_4$. So $|N_{L_2}(u_1)\cap N_{L_2}(u_2)|\leq n-1$, or else a clique on $n$ vertices in $N_{L_2}(u_1)\cap N_{L_2}(u_2)$ together with $\{u_1,u_2,u_3,u_4\}$ induces a $P_4\vee K_n$. Since $u_1$ is an arbitrary vertex in $L_1$ which has two or more nonneighbours in $L_2$, we may assume that $u_1=s_1$. Since $u_3,u_4$ are two arbitrary vertices in $\overline{N_{L_2}}(u_1)$, and $|\overline{N_{L_2}}(u_1)|\geq 3$, we may assume that $u_3,u_4\neq s_2$, then $|N_{L_1}(u_3)\cap N_{L_1}(u_4)|\geq n$, and so a clique on $n$ vertices in $N_{L_1}(u_3)\cap N_{L_1}(u_4)$ together with $\{u_1,u_2,u_3,u_4\}$ induces a $P_4\vee K_n$.
\end{proof}

\begin{claim}\label{clm:case2_4p+2}
Suppose that $L_1,L_2$ are cliques in $G$, such that $|L_1|=|L_2|=\omega(L_1\cup L_2)\geq 4n+2$, Then every vertex in $L_1\setminus L_2$ ($L_2\setminus L_1$) has exactly one nonneighbour in $L_2\setminus L_1$ ($L_1\setminus L_2$).
\end{claim}
\begin{proof}
If $|L_1\cap L_2|\leq n-1$, then it follows from \autoref{clm:case2_3p+3}. So we may assume that $|L_1\cap L_2|\geq n$. Since $\omega(L_1\cup L_2)=|L_2|$, each vertex in $L_1\setminus L_2$ has at least one nonneighbour in $L_2\setminus L_1$. Suppose that $u_1\in L_1\setminus L_2$ has two nonneighbours $u_3,u_4\in L_2\setminus L_1$. Since $G$ is $2K_2$-free and $\omega(L_1\cup L_2)=|L_2|$, there is a vertex $u_2\in L_1\setminus L_2$ which is adjacent to exactly one of $u_3$ and $u_4$, say $u_3$. Therefore, $\{u_1,u_2,u_3,u_4\}$ induces a $P_4$, then $\{u_1,u_2,u_3,u_4\}\cup (L_1\cap L_2)$ contains a $P_4\vee K_n$.
\end{proof}

Let $X_0=\{u\in B_1: u$ and $H\setminus B_1$ are complete.\}, then $\omega(X_0)\geq\max\{g_{n-1}, 4n+2\}$ by \autoref{clm:case2_omegaofcomplete}. Let $X$ be a maximal set such that $X_0\subseteq X\subseteq B_1$ and $\omega(X)=\omega(X_0)$, and $Y=H\setminus X$. By the maximality, for every $u\in Y\cap B_1$ we have $\omega(N_X(u))=\omega(X)$.

\begin{claim}\label{clm:case2_1diff}
Let $L_1\subseteq X$ be a clique on $\omega(X)$ vertices, and $u\in Y\cap B_1$. There is a clique $L_2\subseteq N_X(u)$ on $\omega(X)$ vertices such that $|L_1\cap L_2|\geq \omega(X)-1$.
\end{claim}
\begin{proof}
If $u$ and $L_1$ are complete, then we are done. Now we assume that $u$ and $L_1$ are not complete. Let $L_3\in N_X(u)$ be a clique on $\omega(X)$ vertices. Suppose that $u$ has two nonneighbours $a_1,a_2\in L_1$. By \autoref{clm:case2_4p+2}, there are $a_3,a_4\in L_3$ such that $a_1a_3,a_2a_4\notin E$. Let $K\subseteq L_3\setminus\{a_3,a_4\}$ be a clique on $n$ vertices, then $\{u,a_3,a_2,a_1\}\cup K$ induces a $P_4\vee K_n$.

So $u$ has exactly one nonneighbour in $L_1$, say $a_1$. Let $a_3\in L_3$ such that $a_1a_3\notin E$, then $(L_1\setminus\{a_1\})\cup\{a_3\}$ is a clique on $\omega(X)$ vertices in $N_X(u)$.
\end{proof}

\begin{claim}\label{clm:case2_klcomplete}
Let $L\subseteq Y$ be a clique, then there is a clique $K\subseteq X$ on $\omega(X)$ vertices such that $K$ and $L$ are complete.
\end{claim}
\begin{proof}
If $L\cap B_1=\emptyset$, then $L$ and $X_0$ are complete, and then we are done since $X_0\subseteq X$ and $\omega(X)=\omega(X_0)$. Now we assume that $L\cap B_1\neq\emptyset$. We prove by induction on $|L|$. The claim is true for $|L|=1$ by the maximality of $X$. Now we consider the case that $|L|=m$ ($m\geq 2$), and assume that the claim is true for $|L|\leq m-1$. Let $u_1\in L\cap B_1$. Let $L_1\subseteq X$ be a clique on $\omega(X)$ vertices which is complete to $L\setminus\{u_1\}$. If $u_1$ and $L_1$ are complete then we are done. Now we assume that $u_1$ and $L_1$ are not complete. 
By \autoref{clm:case2_1diff}, there is a clique $L_2\in N_X(u_1)$ on $\omega(X)$ vertices such that $L_1\setminus L_2=\{r_1\}$ and $L_2\setminus L_1=\{r_2\}$, and $r_1r_2,r_1u_1\notin E$. If $r_2$ and $L\setminus\{u_1\}$ are complete then we are done. So assume that $u_2\in L\setminus\{u_1\}$ is nonadjacent to $r_2$, then $\{r_2,u_1,u_2,r_1\}\cup(L_1\cap L_2)$ contains a $P_4\vee K_n$.
\end{proof}

As a corollary of \autoref{clm:case2_klcomplete}, $\omega(H)=\omega(X)+\omega(Y)$. Suppose that $P\subseteq Y$ induces a $P_4$, by \autoref{clm:case2_1diff}, there is a clique in $X$ on $n$ vertices which is complete to $P$. So $Y$ is $P_4$-free and therefore is perfect. Since $X\subseteq B_1\subseteq N(v_1)$, we have that $X$ is $P_4\vee K_{n-1}$-free, and $\chi(X)=\omega(X)$ since $\omega(X)\geq g_{n-1}$. So $\chi(H)\leq\chi(X)+\chi(Y)=\omega(X)+\omega(Y)=\omega(H)$, and then $\chi(G)=\omega(G)$ by \autoref{clm:case2_htog}.
\end{proof}

\section{Conclusions}\label{sec:conc}

\autoref{thm:dividing} is an almost complete characterization for the near optimal colourability for ($H_1,H_2$)-free graphs. The open cases left are that $H_1$ is a forest while $H_2\in \mathcal{X}\setminus\mathcal{X}'$, that is, $H_2 = K_n, K_n-e$ or a paw. Since a graph $G$ is paw-free if and only if each component of $G$ is $K_3$-free or complete multipartite~\cite{Ol88}, the case that $H_2$ is a paw can be reduced to the case that $H_2=K_3$. Clearly, the family of ($H_1,K_n$)-free graphs is near optimal colourable for every $n$ if and only if the family of $H_1$-free graphs is $\chi$-bounded. So the Gy{\'a}rf{\'a}s conjecture is equivalent to that each graph family of ($H_1, H_2$)-free graphs with $H_1$ being a forest and $H_2 = K_n$ is near optimal colourable. The other open case is that $H_2 = K_n-e$ with $n\geq 4$. Since $K_n-e$ is an induced subgraph of $P_4\vee K_{n-2}$, we conclude that the family of ($2K_2, K_n-e$)-free graphs is near optimal colourable by \autoref{thm:2k2p4kp}. By \autoref{lem:2k23k1}, it suffices to consider the case that $H_1$ is a forest with independent number at least 3 and $H_2=K_n-e$.
%We list the other open case as below.

%Since $K_n$ is an induced subgraph of $K_{n+1}-e$, we believe that Case 2 is as difficult as Case 1. As such, we list Case 2 as an open problem below.

%Let $\mathcal{G}$ be the family of ($H_1,H_2$)-free graphs. If we cannot decide whether $\mathcal{G}$ is near optimal colourable by \autoref{thm:dividing}(1)(2)(3), then by \autoref{lem:2k23k1}, it must be the case that one of $H_1$ and $H_2$ is a forest while the other is a complement of a linear forest, and neither $H_1$ nor $H_2$ is an induced subgraph of $P_4$. Since $H_1$ and $H_2$ are exchangeable, we may assume that $H_1$ is a forest, and $H_2$ a complement of a linear forest, and then we may use \autoref{thm:dividing}(4). We cannot   decide whether $\mathcal{G}$ is near optimal colourable only if $H_2$ is isomorphic to one of $K_n$, $K_n-e$ and $(K_2+K_1)\vee K_1$ for some $n$, as the following three cases:

%\case{1} $H_1$ is a forest and is not an induced subgraph of $P_4$. $H_2=K_n$ for some $n\geq 3$.

%\case{2} $H_1$ is a forest and is not an induced subgraph of $P_4$. $H_2=K_n-e$ for some $n\geq 4$.

%\case{3} $H_1$ is a forest and is not an induced subgraph of $P_4$. $H_2$ is a paw.

\begin{problem}\label{prob:decidecase2}
Decide whether the family of ($H_1,H_2$)-free graphs is near optimal colourable when $H_1$ is a forest with independent number at least 3 and $H_2=K_n-e$.
\end{problem}

Gy{\'a}rf{\'a}s conjecture is a major open problem in graph colouring, and only few partial results are known. 
%For example, Gravier, Ho{\` a}ng, and Maffray \cite{GHM03} proved that every $P_t$-free graph $G$ for $t\geq 4$ with $\omega(G)\geq 2$ has $\chi(G)\leq (t-2)^{\omega(G)-1}$. 
See~\cite{SR19, SS20} for more results on the Gy{\'a}rf{\'a}s conjecture. Since $K_n$ is an induced subgraph of $K_{n+1}-e$, we believe that \autoref{prob:decidecase2} is as difficult as the Gy{\'a}rf{\'a}s conjecture.
Our results on ($P_6$,diamond)-free graphs~\cite{GHJM23} solves a subproblem of \autoref{prob:decidecase2}.

\noindent {\bf Acknowledgement.} We thank Dani{\"e}l Paulusma for pointing to us reference \cite{DP18}.

\end{document}